\documentclass[12pt]{article}
\usepackage{amsfonts}
\usepackage{amsmath, amsthm, upref}
\usepackage{enumerate}
\usepackage[parfill]{parskip}
\usepackage{color}
\usepackage{hyperref}
\usepackage{multimedia}
\usepackage{graphics}
\usepackage{epsfig}
\usepackage{graphicx}
\usepackage[greek,english]{babel}
\usepackage{textcomp}
\usepackage{eurosym}

\newcommand{\gl}{\lambda}

\newcommand{\mbe}{\mathbb{E}}
\newcommand{\mbf}{\mathbb{F}}
\newcommand{\mbg}{\mathbb{G}}
\newcommand{\comment}[1]{}

\newcommand{\gD}{\Delta}

\newcommand{\mbh}{\mathbb{H}}
\newcommand{\mbr}{\mathbb{R}}

\newcommand{\mcf}{\mathcal{F}}
\newcommand{\mce}{\mathcal{E}}
\newcommand{\mcg}{\mathcal{G}}
\newcommand{\mch}{\mathcal{H}}

\newcommand{\mct}{\mathcal{T}}
\newcommand{\mcc}{\mathcal{C}}

\newcommand{\ga}{\alpha}
\newcommand{\gs}{\sigma}
\newcommand{\gL}{\Lambda}
\newcommand{\gb}{\beta}

\newtheorem{theorem}{Theorem}

\newtheorem{definition}{Definition}

\newtheorem{lemma}{Lemma}

\newtheorem{proposition}{Proposition}
\newtheorem{remark}[theorem]{Remark}

\newcommand{\bee}{\begin{equation}}
\newcommand{\eee}{\end{equation}}
\newcommand{\bea}{\begin{eqnarray}}
\newcommand{\eea}{\end{eqnarray}}
\newcommand{\bean}{\begin{eqnarray*}}
\newcommand{\eean}{\end{eqnarray*}}

\newcommand{\gep}{\varepsilon}

\newcommand{\mctl}{\mct(\gL)}

\newcommand{\norm}{\parallel}

\newcommand{\rawi}{\rightarrow\infty}

\newcommand{\mcn}{\mathcal{N}}

\begin{document}

\title{ Strict Local Martingales With Jumps}
\author{ Philip Protter\thanks{
Statistics Department, Columbia University, New York, NY 10027; supported in
part by NSF grant DMS-1308483} \\
\\
\\
}
\date{\today }
\maketitle

\begin{abstract}

A strict local martingale is a local martingale which is not a martingale. There are few explicit examples of ``naturally occurring'' strict local martingales with jumps available in the literature. The purpose of this paper is to provide such examples, and to illustrate how they might arise via filtration shrinkage, a phenomenon we would contend is common in applications such as filtering, control, and especially in mathematical finance. We give a method for constructing such examples and analyze one particular method in detail.

\end{abstract}

\section{Introduction}\label{s1}

It has recently been remarked in the literature that there is a paucity of examples of strict local martingales with jumps (cf,~\cite{FJS},\cite{KKN}), aside from the case treated by Chybiryakov~\cite{OC}. Examples of strict local martingales with jumps are also given in~\cite{KKN}. And, as is remarked in~\cite{KKN}, if $M$ is a continuous strict local martingale, and $N$ is a purely discontinuous martingale, then the sum $L=M+N$ is again a strict local martingale, but this seems an artificial way of creating examples of strict local martingales with jumps, and the purpose here is to show how they arise naturally, to supplement the theory of~\cite{OC}. Aspects of the importance of strict local martingales arising in quite different theories is emphasized in the now classic papers of Elworthy, Li, and Yor~\cite{ELY}, Delbaen and Schachermayer~\cite{DS}, and more recently Madan and Yor~\cite{MY}.

The obvious candidate for a strict local martingale with jumps is that of a L\'evy process. However no L\'evy processes can be a strict local martingale (see~\cite{JacodS} or~\cite[Exercise 29, p. 49]{PP}). While of course strict local martingales with jumps can arise as stochastic integrals with respect to L\'evy processes, given a candidate integrand it is hard to determine if the stochastic integral that will result will be a strict local martingale, or just a martingale. 

The recent interest in the study of models of financial bubbles has gone a long way to illustrate the intrinsic interest in producing examples of ``naturally occurring'' strict local martingales with jumps. Our goal is to propose a method doing that.

The author wishes to thank the National University of Singapore, where much of the work on this paper was done, as well as discussions with Bob Jarrow and his office mate at the NUS Freddy Delbaen. He also benefited from many discussions with Yan Zeng back in 2005.

Our approach in this paper is to take a strict local martingale $Y$ with continuous paths on a given filtered probability space $(\Omega,\mcg,P,\mbg)$ where $\mbg=(\mcg_t)_{t\geq 0}$, and to project it onto a significantly smaller filtration $\mbf$ to get a new process $X$. We take the projection in such a way that $X$ is a strict local martingale in the smaller filtration. If the filtration is sufficiently ``poor" then it will lose many of the $\mbg$ stopping times, and some of the remaining $\mbf$ stopping times will become totally inaccessible. (We recall in Section~\ref{s2} the definition of totally inaccessible stopping times.) These totally inaccessible stopping times are potential jump times of the strict local martingale $X$. Since we have a way to construct continuous strict local martingales at will due to the work of Delbaen and Shirakawa~\cite{DelbaenShirakawa}, this in turn gives us a way to construct strict local martingales with jumps, also at will. It also illustrates how they can arise naturally in practice as deriving from continuous strict local martingales when information is restricted. 

In Section~\ref{s2} we present a method of filtration shrinkage that is simple and creates strict local martingales with jumps. We also establish some preliminary results in Section~\ref{s2} and continue with more in Section~\ref{s3}. Section~\ref{s4} contains the main results of the paper, and includes an analysis of the compensators of the jumps under supplemental hypotheses. 

\subsection*{Acknowledgements} The author wishes to thank two anonymous referees for their constructive comments that have led to an improvement in the clarity of the presentation of these results. 

\section{The Basic Framework}\label{s2}

Let us assume given a filtered complete probability space $(\Omega,\mcg,P,\mbg)$ satisfying the usual conditions, and which is rich enough to support a Brownian motion. To analyze the framework we develop, we need a concept taken from a paper of Jacod and Skorokhod~\cite{JS}, that of the jump of a filtration.

\begin{definition}\label{d1}
Let $S\leq T$ be two stopping times for a given filtration $\mbh$. We say that the filtration $\mbh$ \emph{jumps from $S$ to $T$} if for all $t\geq 0$ the $\gs$ algebras $\mch_S$ and $\mch_t$ coincide up to null sets on $\{S\leq t<T\}$.
\end{definition}

We wish to create a subfiltration $\mbf$, with $\mcf_t\subset\mcg_t$ for all $t\geq 0$. One way to do this is the following. Let $X$ be an adapted process with continuous sample paths (for example, $X$ could be a Brownian motion). We let $\tau_x$ denote the first passage time for the level $x$. That is, 
\bee\label{e1}
\tau_x=\inf\{t>0:X_t\geq x\}.
\eee
Next we let $\gL\subset (0,\infty)$, and then $\mct(\gL)=(\tau_x)_{x\in\gL}$ denotes the collections of all first passage times of $X$ for levels $x\in\gL$. We let $\mbf$ be the minimal filtration such that every $\tau_x, x\in\gL$ is a stopping time, and which satisfies the usual hypotheses. That is, $\mbf$ is given by $\mcf_t=(\cap_{u>t}\mcf_u^0)\vee\mcn$, for each $t\geq 0$, where $\mcn$ denotes all of the $P$ null sets of $\mcg$, and where
\bee\label{e2}
\mcf_t^0=\gs(\tau_x\leq s, s\leq t,x\in\gL)
\eee

\begin{definition}\label{d2}
We say a point $\gb\in\gL$ is \emph{isolated from below in $\gL$} if $\sup_{\ga\in\gL}\{\ga<\gb\}<\gb$.
\end{definition}

We can define this a little more abstractly, by considering a family of positive random variables $(T_\ga)_{\ga\in\gL}$, and we call this family \emph{totally ordered} if for any $\ga\neq\gb\in\gL$, either $T_\ga<T_\gb$ or $T_\gb<T_\ga$. Propositions~\ref{t1} and~\ref{t2} are taken from the 2006 PhD thesis of Yan Zeng~\cite{Zeng}.

\begin{proposition}\label{t1}
Let $\mct(\gL)=(T_\ga)_{\ga\in\gL}$ be a totally ordered family, and let $\mbf$ be the minimal filtration satisfying the usual hypotheses and making all of $\mctl$ stopping times. Fix a time $T_\gb\in\mctl$, and let $S=\sup_{\ga\in\gL}\{T_\ga: T_\ga<T_\gb\}$. 
If $S<T_\gb$ a.s. then $\mbf$ jumps from $S$ to $T_\gb$. 
\end{proposition}

\begin{proof}
Note that $S$ is a stopping time, since we can also express $S$ as the supremum of a countable number of stopping times. Therefore we can work with the $\gs$ algebra $\mcf_S$. Then $\mcf_S\cap\{S\leq t\}\subset\mcf_t$, and $\mcf_S\cap\{S\leq t<T_\gb\}\subset\mcf_t\cap\{S\leq t<T_\gb\}$.

Conversely, for any $\gamma\in\gL$, either $T_\gamma>S$ a.s., or $T_\gamma\leq S$ a.s. For $s\leq t$ if $T_\gamma\leq S$ then $T_\gamma\in\mcf_S$ and
$$
\{T_\gamma\leq s\}\cap\{S\leq t<T_\gb\}\in\mcf_S\cap\{S\leq t<T_\gb\}.
$$
If $T_\gamma> S$ then $T_\gamma\geq T_\gb$ a.s., and $\{T_\gamma\leq s\}\cap\{S\leq t<T_\gb\}=\emptyset$ for $s\leq t$. In any event, we have $\{T_\gamma\leq s\}\cap\{S\leq t<T_\gb\}\in\mcf_S\cap\{S\leq t<T_\gb\}.$ Since $\mcf_t=\gs(\{T_\gamma\leq s\}: s\leq t,\gamma\in\gL)$, by the monotone class theorem we conclude that $\mcf_t\cap\{S\leq t<T_\gb\}\subset\mcf_S\cap\{S\leq t<T_\gb\}$.
\end{proof}

Before we state the next proposition, let us recall the definition of a totally inaccessible stopping time. Recall that a \emph{predictable stopping time} $T$ is a stopping time such that there exists an announcing sequence of stopping times $T^n$ with $T^n\leq T^{n+1}$ a.s., $T^n<T$ a.s. each $n$, and $\lim_{n\rawi}T^n=T$ a.s. 

A \emph{totally inaccessible time} is a stopping time $U$ such that for any predictable time $T, P(T=U)=0$. 

A renowned theorem of P.A. Meyer states that with respect to the natural, completed filtration of a  a Feller strong Markov process, the collection of totally inaccessible stopping times coincides with the  collection of jump times of the process (this theorem is stated carefully in, for example,~\cite{PP}).

\begin{proposition}\label{t2}
Suppose $\mbf$ is the minimal filtration that satisfies the usual hypotheses and renders each $T_\ga$ a stopping time. For any $\gb\in\gL$, $T_\gb$ is totally inaccessible if and only if (i) $\gb$ is isolated from below in $\gL$ and (ii) the regular conditional distribution of $T_\gb$ given $\mcf_S$ is continuous. 
\end{proposition}

\begin{proof}
First we assume $T_\gb$ is totally inaccessible. Suppose $\gb$ is not isolated from below in $\gL$. Then there exists a sequence $\ga_n\nearrow\beta$, and the stopping times $T_{\ga_n}\nearrow T_\gb$ a.s., and  $T_{\ga_n}<T_\gb$ a.s. for each $n$. Thus $T_\gb$ is predictable, a contradiction. Moreover since $T_\gb$ is totally inaccessible, its compensator is continuous. Thus such a sequence $\ga_n\nearrow\beta$ cannot exist, which implies that $S<T_\gb$ a.s. 

By Proposition~\ref{t1} $\mbf$ jumps from $S$ to $T_\gb$ and it is therefore simple to check (see~\cite[Theorem 4, p.20]{Zeng}) that its compensator is 
\bee\label{e3}
\int_0^t1_{(S,T_\gb]}(u)\frac{P(T_\gb\in du\vert\mcf_S)}{P(T_\gb\geq u\vert\mcf_S)}
\eee
and since $T_\gb$ is totally inaccessible by hypothesis, it must be continuous.

For the converse, we wish to show $T_\gb$ is totally inaccessible. Since by hypothesis $\gb$ is isolated from below in $\gL$, we have $S<T_\gb$ by the path continuity of $X$.  By Proposition~\ref{t1} and~\eqref{e3} we have that when $P(T_\gb\in du\vert\mcf_S)$ is continuous then $T_\gb$ has a continuous compensator. We conclude that $T_\gb$ is totally inaccessible. 
\end{proof}

Note that it is not possible at this level of generality to conclude that the compensators of times such as $T_\gb$ where $\gb$ is a left endpoint of $\gL$ are absolutely continuous with respect to Lebesgue measure. One way to achieve such results is to relate them to a strong Markov process (eg, a Hunt process). This approach is pursued in~\cite{GZ} and in more generality in~\cite{JMP}. But we hasten to add that if, for example, the underlying continuous stochastic process $X$ is a standard Brownian motion, then one can verify that all compensators of the times $T_\gb$ of left isolated points of $\gL$ do indeed have absolutely continuous paths. 

\begin{remark}\label{r1} The idea of developing a theory of filtration shrinkage is not new. The seminal work in this area is the 1978 paper of Br\'emaud and Yor~\cite{BY}. The idea of shrinking the filtration by considering first passage times corresponding to all $x\in\gL$ is not new and was considered by A. Deniz Sezer in her thesis, resulting in the papers~\cite{Deniz} and~\cite{JPSezer}. Her approach was embedded within the theory of Markov processes and excursion theory, involving homogeneous random sets. In the second paper~\cite{JPSezer}, her techniques are applied to resolve questions in the theory of credit risk, and issues of martingale representation are broached. Such techniques could also be used in this framework (under additional assumptions), but we do not attempt to do that here. Recent papers using the idea of filtration shrinkage include~\cite{FP},\cite{JarrowP}, and~\cite{Martin}.
\end{remark}

\section{Theory Preliminaries}\label{s3}

In this section we develop and recall some theory that we will use to establish our results in Section~\ref{s4}, the latter being the heart of the paper.

We first consider the case of \emph{nonnegative} strict local martingales. These arise often in finance, and are connected to the existence of models of financial bubbles (see, e.g.,~\cite{CH},\cite{JPShimbo},\cite{BFN},\cite{PhilipToBe}). Delbaen and Shirakawa~\cite{DelbaenShirakawa} have given a method of constructing examples of continuous strict local martingales as solutions of stochastic differential equations (see also~\cite{MU} and~\cite{KL}), hence it is reasonable to begin by supposing that $Z$ is a $\mbg$ continuous nonnegative strict local martingale. Let $U$ be an arbitrary continuous $\mbg$ adapted process, and let $\gL$ be as before: $\gL\subset [0,\infty)$ and we assume $\gL$ has left isolated points, and the left isolated points contain a sequence tending to $\infty$. We let
\bee\label{e9}
\tau_x=\inf\{t>0:U_t\geq x\}.
\eee
We define  $\mbf$ by
\bee\label{e10}
\mcf_t=\cap_{u>t}\mcf^0_u\text{ where }\mcf^0_t=\gs(\tau_x\leq s, s\leq t,x\in\gL),
\eee
where the $P$ null sets are added to all $\mcf_t^0$. Note that $\mbf$ satisfies the ``usual hypotheses."

\begin{proposition}\label{t4}
Suppose the quadratic variation $t\mapsto [Z,Z]_t$ is always strictly increasing, a.s.If the reducing stopping times of $Z$ are also stopping times in $\mbf$, then the optional projection $M$ of $Z$ onto $\mbf$ is an $\mbf$ strict local martingale. Moreover if $U=Z$ then $M$ has jumps at every time $\tau_\gb$, where $\gb$ is a left isolated point of $\gL$.
\end{proposition}

\begin{proof}
Since the reducing stopping times of the $\mbg$ strict local martingale $Z$ are stopping times in $\mbf$, it follows from Theorem 11 of~\cite{FP} that $M$ is an $\mbf$ local martingale. Since $Z$ is a nonnegative strict local martingale, we know that $t\mapsto E(Z_t)$ is decreasing. Choose an arbitrary point $\gb\in\gL$ that is isolated from below, and let $S$ be defined as before:
\bee\label{e11}
S=\sup_{\ga\in\gL}\{\tau_\ga: \tau_\ga<\tau_\gb\}
\eee
It remains to study the jumps of $M$ at the left isolated points of $\gL$. Note that we know by Theorem~\ref{t2} these are the totally inaccessible times of $\mbf$. We have that $\gD M_{\tau_\gb}=M_{\tau_\gb}-M_{\tau_\gb-}=M_{\tau_\gb}-M_S$, where $S=\sup_{\ga\in\gL}\{\tau_\ga: \tau_\ga<\tau_\gb\}$, since the filtration $\mbf$ jumps from $S$ to $\tau_\gb$ by Theorem~\ref{t1}. But $M_{\tau_\gb}=E(Z_{\tau_\gb}\vert\mcf_{\tau_\gb})=E(\gb\vert\mcf_{\tau_\gb})=\gb$, and $M_{S}=E(Z_{S}\vert\mcf_{S})=E(\ga_0\vert\mcf_{S})=\ga_0$, where $\ga_0=\sup\{x\in\gL\text{ such that }x<\gb\}$.
\end{proof}

We now turn from considering only nonnegative strict local martingales to general strict local martingales. We know already by Theorem 11 of~\cite{FP} that the optional projection of a $\mbg$ strict local martingale $Z$ onto a subfiltration $\mbf$ is again a local martingale, as long as a reducing sequence of stopping times in $\mbg$ is also a sequence of $\mbf$ stopping times. So the only issue is whether or not $Z$ being strict in $\mbg$ also implies that $Z$ is strict in $\mbf$. 

The general case follows from the positive case if we use the Krickeberg Decomposition for Local Martingales, established by Kazamaki in 1972~\cite{Kaz}. Let $\mct$ denote the collection of a.s. finite stopping times. For a local martingale $X$ we define a 1 norm by 
$$
\parallel X\parallel_1=\sup_{\{T\in\mct\}}E(\vert X_{T}\vert).
$$

\begin{theorem}[Krickeberg Decomposition for Local Martingales]\label{Krick}
Let $X$ be a local martingale.  Then $\parallel X\parallel_1=\sup_nE(\vert X_{T_n}\vert)$, for every reducing sequence of stopping times $(T_n)$ with $\lim_{n\rightarrow \infty}T_n=\infty$.  If $\parallel X\parallel_1<\infty$, then there exist two positive local martingales $X^p$ and $X^n$ such that 
\bee\label{e12}
X=X^p-X^n
\eee
\bee\label{e13}
\parallel X\parallel_1=\parallel X^p\parallel_1+\parallel X^n\parallel_1
\eee
The Krickeberg decomposition above is unique as long as one insists on~\eqref{e13}. Moreover, one can choose a reducing sequence of stopping times $(T_n)_{n\geq 1}$ that simultaneously reduces all three of $X$, $X^1$, and $X^2$.
\end{theorem}
We note that the last statement in Theorem~\ref{Krick} is not contained in Kazamaki's original paper, but it is simple to check directly.

\begin{proposition}\label{t5}
Let $Y$ be a strict local martingale on a space $(\Omega,\mcg,P,\mbg)$ and let $\mbf$ be a subfiltration of $\mbg$. Assume there exists a sequence of stopping times $(T_n)$ in $\mbf$ that form a reducing sequence for $Y$ in $\mbg$. Also assume that 
$$
\parallel Y\parallel_1=\sup_nE(\vert Y_{R_n}\vert)<\infty
$$
for every reducing sequence $(R_n)_{n\geq 1}$ in $\mbg$. Then $M=^oY$ is a strict local martingale in $\mbf$ with a reducing sequence for $M$ being $(T_n)_{n\geq 1}$.
\end{proposition}
We emphasize that the fact that $M=^oY$ is a local martingale is established in~\cite{FP}. The novelty in Theorem~\ref{t5} is that it is again strict.

\begin{proof}
By assumption we have $\parallel Y\parallel_1<\infty$. Let $Y=Y^p-Y^n$ be its Krickeberg decomposition. If $Y^p$ and $Y^n$ are both martingales (ie, not strict local martingales), then so is $Y$, and this violates our assumption that it is a strict local martingale. So at least one of $Y^p$ and $Y^n$ is a strict local martingale. Let us assume that $Y^p$ is strict. 

Since $Y^p$ is a positive strict local martingale, its expectation must be decreasing, and in particular cannot be constant. So its projection onto $\mbf$, call it $M^p=^oY^p$, is also a strict local martingale, because it is positive, has decreasing expectation, and the $\mbf$ reducing sequence $(T_n)_{n\geq 1}$ is a reducing sequence for $Y^p$, by assumption and Theorem~\ref{Krick}.  We know also that $M^n$ is a local martingale. It does not matter if it is a martingale or a strict local martingale, in either case the difference $M^p-M^n$ is a strict local martingale because $M^p$ is one. Note that the reasoning is identical had we chosen $Y^n$ to be strict, instead of $Y^p$.
\end{proof}

\section{Main Results}\label{s4}

We will obtain examples of strict local martingales with jumps by projecting continuous strict local martingales onto a subfiltration. Therefore let us first briefly review the situation for continuous processes. 

There are several examples of {continuous} strict local martingales.  The easiest and perhaps most useful family known to date is that of Delbaen and Shirakawa~\cite{DelbaenShirakawa}. See alternatively Mijatovic and Urusov~\cite{MU}. They consider solutions of stochastic differential equations of the form
\bee\label{e5}
dX_t=\gs(X_t)dB_t,\qquad X_0=1,
\eee
where $B$ is standard Brownian motion. Under the condition that 
\bee\label{e6}
\int_0^\gep\frac{x}{\gs(x)^2}ds=\infty
\eee
 the solution $X$ of~\eqref{e5} is strictly positive ($\gep>0$). They show that $X$ is a strict local martingale if and only if
 \bee\label{e7}
 \int_\gep^\infty\frac{x}{\gs(x)^2}ds<\infty.
 \eee 
An example is the famous inverse Bessel(3) process of Johnson and Helms~\cite{JH}, which corresponds to $\gs(x)=-x^2$. 

We next establish a preliminary result that we find useful later. Let $B$ be a standard Brownian motion on a space $(\Omega,\mcg,P,\mbg)$ and let
\bee\label{e1bis}
\tau_x=\inf\{t>0:B_t\geq x\}\text{ and }\nu_x=\inf\{t>0:B_t\leq x\}
\eee
Let $\gL\subset (0,\infty)$, and let $\mbf$ be given by
\bee\label{e2bis}
\mcf^0_t=\gs(\tau_x\leq s, s\leq t,x\in\gL)\vee\gs(\nu_x\leq s, s\leq t,x\in\gL)
\eee
and $\mcf_t$ is $\cap_{u>t}\mcf^0_t\vee\mcn$, where $\mcn$ are the $P$ null sets of $\mcg$.
\begin{eqnarray}\label{e2ter}
\text{ Assume that }&\gL&\text{ contains a sequence of points }\ga_n\text{ with }\limsup\ga_n=\infty\notag\\
&&\text{ and }\liminf\ga_n=-\infty.
\end{eqnarray}
\begin{theorem}\label{t2bis}
Let $B$ be Brownian motion and let $\mbf$ be as given in~\eqref{e2bis} above, made right continuous. Let $X$ be a solution of an SDE of the form
\bee\label{e7bis}
X_t=x+\int_0^t\gs(X_s)dB_s
\eee
where $\gs\in\mcc^2, \gs>0$, and both $\gs$ and $\gs\gs^\prime$ are Lipschitz continuous. Then there exists a sequence of $\mbf$ stopping times $(T_n)_{n\geq 1}$ increasing to $\infty$ a.s. such that $X_{t\wedge T_n}$ is a bounded process, for each $n$.
\end{theorem}

\begin{proof}
We let $R_n$ be the first passage time for $B$ of $\ga_n$, where $\ga_n$ are given in~\eqref{e2ter}. Next let $W=-B$ and let $S_n$ be the first passage time for $W$ of $\ga_n$. If we next let $T_n=R_n\wedge S_n$, we have $\sup_{s\leq T_n}\vert B_s\vert\leq\ga_n$, and hence is bounded.

We use an old theorem of H. Doss~\cite{Doss}. See also~\cite[Theorem 25, p. 289]{PP}. In particular it is shown in~\cite{PP} that if 
\bee\label{e7a}
dY_t=\frac{1}{2}\gs(Y_t)\gs^\prime(Y_t)dt+\gs(Y_t)dB_t; \quad Y_0\in\mbr,
\eee
then 
\bee\label{e7ter}
Y_t=h^{-1}(B_t+h(Y_0))
\eee
where 
$$
h(y)=\int_{y_0}^y\frac{1}{\gs(u)}du; \qquad \gep>0
$$
Since $\gs>0$ we have $h$ is strictly increasing, and therefore $h^{-1}$ is continuous; therefore our times $T_n$ also bound $Y$ as a consequence of the representation~\eqref{e7ter}. We next use a change of measure technique. 

We define $Q$ by letting $Z=\frac{dQ}{dP}$ where the martingale (in the $\mbg$ filtration) $Z_t=E(Z\vert\mcg_t)$ is given by
\bee
Z_t=1+\int_0^tZ_s(\frac{1}{2}\gs(X_s)\gs^\prime(X_s))dB_s
\eee
Then $Q$ is equivalent to $P$, and under $Q$, the process $X$ satisfies the SDE
$$
dX_t=\gs(X_t)dB_t+\frac{1}{2}\gs(X_t)\gs^\prime(X_t)ds
$$ 
by Girsanov's theorem. We use the sequence $(T_n)_{n\geq 1}$ of selected first passage times of $\vert B\vert$  to see that $X$ is locally bounded under $Q$ with stopping times in $\mbf$, and since $P\ll Q$ we have that $X$ is locally bounded under $P$ as well using the same stopping times $(T_n)_{n\geq 1}$.

\end{proof}

We now have established enough results to deduce the following:

\begin{theorem}\label{t3}
Let $B$ be a standard Brownian motion on its canonical space $(\Omega,\mcg,P,\mbg)$, and let $X$ be a solution of~\eqref{e5} satisfying~\eqref{e6} and~\eqref{e7}.  Let $\gL\subset [0,\infty)$ containing left isolated points such that $\sup\{\eta: \eta\in\gL\}=\infty$, and let $\mbf$ be the filtration generated by the first passage times of $B$ for every point in $\gL$, with $\mbf$ satisfying the usual hypotheses. Let $M$ be the projection of $X$ onto $\mbf$. That is, $M=^oX$, and for a fixed $t$ one has $^oX_t=E(X_t\vert\mcf_t)$. We take the c\`adl\`ag version of $M$. Then $M$ is a strict local martingale that jumps at every time $T_\gb$, for $\gb$ a left isolated point in $\gL$, and for $T_\gb$ the first passage time of $B$ at $\gb$.

\end{theorem}

\begin{proof}
Note that the notation $^oX$ refers to the optional projection of $X$ onto the filtration $\mbf$. We let $\ga_n$ be a sequence of points in $\gL$ such that $\ga_n\rightarrow\infty$, and let $R_n, S_n$ and $T_n=R_n\wedge S_n$ be as given in the proof of Theorem~\ref{t2bis}, and then by this same Theorem~\ref{t2bis} we have that $X_{t\wedge T_n}$ is bounded. But if it is bounded, so also is its optional projection, since the times $(T_n)_{n\geq 1}$ are not just $\mbg$ stopping times, but are also $\mbf$ stopping times. Therefore $M=^oX$ is locally bounded by the $\mbf$ stopping times $T_n$.

We observe that $E(X_t)=E(M_t)$ for each $t\geq 0$, and that $t\mapsto E(X_t)$ is decreasing because $X$ is a nonnegative strict local martingale; therefore $t\mapsto E(M_t)$ is also decreasing, and hence $M$ cannot be a martingale, and is therefore a strict local martingale. 

It remains to observe that we know the jump times of $M$  are the totally inaccessible times of $\mbf$. This is by Theorem~\ref{t2}. We have that $\gD M_{T_\gb}=M_{T_\gb}-M_{T_\gb-}=M_{T_\gb}-M_S$, where $S=\sup_{\ga\in\gL}\{T_\ga: T_\ga<T_\gb\}$, since the filtration $\mbf$ jumps from $S$ to $T_\gb$ by Theorem~\ref{t1}. 
\end{proof}

In Theorem~\ref{t3} we let $X$ be the solution of the SDE~\eqref{e5} and we projected it onto $\mbf$, and we let $M=^oX$. $M$ is a local martingale, and thus it can written uniquely as
\begin{equation}\label{e14}
M=M^c+M^d
\end{equation}
where $M^c$ denotes its continuous part and $M^d$ denotes its ``purely discontinuous'' part. This was originally shown by C. Yoeurp~\cite{Yoeurp}. We then know that if $M$ jumps at a stopping time $R$, the process $1_{\{t\geq R\}}$ has a compensator $A^R_t$ such that $1_{\{t\geq R\}}-A^R_t$ is a martingale, and that if $M^d$ is locally square integrable it can be written as a compensated sum of jumps (cf~\cite[p. 266]{Meyer}), so that:
\begin{equation}\label{e15}
M^d_t=\sum_{\gb\text{ left isolated in }\gL}\left(\gD M_{T_\gb}1_{\{t\geq T_\gb\}}-C^{\gb}_t\right)
\end{equation}
where $C^{\gb}_t$ is the compensator of the process $\gD M_{T_\gb}1_{\{t\geq T_\gb\}}$. 

As was indicated in~\cite{JMP} it is of interest to give conditions when the compensator process $A^\gb$ has absolutely continuous paths; that is, when $A^\gb$ is of the form $A^\gb_t=\int_0^th(\gb)_sds$, where $h(\gb)$ is an adapted stochastic processes, \emph{a priori} changing with each time $T_\gb$. In this case the process $h(\gb)$ has a natural interpretation as an instantaneous (stochastic) relative likelihood that a jump will occur in the time interval $(t,t+dt)$. It is a common assumption in research papers that such compensators are indeed absolutely continuous, but it is rather hard in practice to prove that they indeed are. 

We wish to give conditions that ensure that the compensators of the time $T_\gb$ are indeed absolutely continuous. Note that this will also imply that the compensators $C^\gb$ of equation~\eqref{e15} are also absolutely continuous. To see this, we give a proof in the locally bounded case, and then by implicit localization, we assume it is bounded. Suppose $H=\gD M_T$ is bounded in $L^\infty$. Let $b\mcf_s$ denote the bounded $\mcf_s$ measurable random variables. We then have, for $J\in b\mcf_s$ and $s<t$: 
\begin{eqnarray}\label{e15bis}
\vert E((C^\beta_t-C^\beta_s)J))\vert&=&\vert E(J(H1_{\{t\geq T\}}-H1_{\{s\geq T\}}))\vert=\vert E(JH1_{\{s<T\leq t\}})\vert\\
&\leq& \norm H\norm_{L^\infty}\vert E(J1_{\{s\leq T\leq t\}})\vert\leq \norm H\norm_{L^\infty} E(J\int_s^t\vert h(\gb)_r\vert dr)\notag
\end{eqnarray}
and then it remains only to apply Yan Zeng's extension of the Ethier-Kurtz Criterion for the absolute continuity of compensators (see~\cite[Theorem 2]{JMP}). Note that two applications of the Cauchy-Schwarz inequality and using almost the same argument proves the case for $H=\gD M_T$ locally in $L^2$.

We begin with a preliminary result, which uses special properties of Brownian motion.

\begin{theorem}\label{t6}
For a Brownian motion $B$ on the space $(\Omega,\mcg,P,\mbg)$ and a set $\gL\subset\mbr_+$, we let $T_\gb$ be the first passage time of $B$ for a level $\gb\in\gL$. We define $\mbf$ as we did before (ie, see~\eqref{e2}). For a left isolated point $\gb\in\gL$ we have that the compensator of $T_\gb$ is absolutely continuous. That is, there exists an $\mbf$ adapted  process $\gl^\gb=(\gl^\gb_s)_{s\geq 0}$ such that $1_{\{t\geq T_\gb\}}-\int_0^t\gl^\gb_sds$ is an $\mbf$ martingale. 
\end{theorem}

\begin{proof}
We fix a left isolated point $\gb\in\gL$, and we let $N_t=1_{t\geq T_\gb}$. We let $\mbe\subset\mbf\subset\mbg$ be the filtration 
$$
\mce_t=\gs(\tau_x\leq s, s\leq t,x\in\gL\cap[0,\gb]).
$$
It is simple to show via an application of the monotone class theorem that 
\begin{equation}\label{e16}
\gs(\mcf_t\cap\{T_\gb>t\})=\gs(\mce_t\cap\{T_\gb>t\}).
\end{equation}
By another application of the monotone class theorem we have the following identity:
\begin{equation}\label{e17}
P(T_\gb>t+h\vert\gs(\mce_t\cap\{T_\gb>t\})P(T_\gb>t\vert\mce_t)=P(T_\gb>t+h\vert\mce_t)1_{\{T_\gb>t\}}
\end{equation}
For our fixed point $\gb\in\gL$ we let $\ga=\sup\{x\in\gL: x<\gb\}$ with the convention that $\sup\emptyset=0$. It follows that $T_\ga=\sup_{x\in\gL\cap [0,\gb]}T_x$ and is a stopping time for the $\mbg$ filtration. We also observe that 
$$
\mce_t\subset\vee_{x\in(\gL\cap [0,\gb])}\mcg_{T_x\wedge t}\subset\mcg_{T_\ga\wedge t}\subset\mcg_{T_\ga}.
$$
Therefore we have 
$$
1_{\{T_\gb>t\}}=E(E(1_{\{T_\gb>t\}}\vert\mcg_{T_\ga})\mce_t).
$$
Next we use that the process $(T_x)_{x\geq 0}$ has independent and stationary increments for the time changed original filtration $\mcg_{t_x}$ to get
\begin{eqnarray}\label{e17}
E(E(1_{\{T_\gb>t\}}\vert\mcg_{T_\ga})\vert\mce_t)&=&E(E(1_{\{T_\gb-T_\ga>t-T_\ga\}}\vert\mcg_{T_\ga})\vert\mce_t)\notag\\
&=&E(1-F_{T_\gb-T_\ga}(t-T_\ga)\vert\mce_t)\notag\\
&=&1_{\{T_\ga>t\}}+\{1-F_{T_\gb-T\ga}(t-T_\ga)\}1_{\{T_\ga\leq t\}}
\end{eqnarray}
where $F_{T_\gb-T_\ga}(x)=F_{T_{\gb-\ga}}(x)$ is the distribution function of $T_\gb-T_\ga$, and moreover it equals $\int_0^xf_{\gb-\ga}(u)du$, where
\begin{equation}\label{e17bis}
f_\gamma(u)=\gamma(2\pi u^3)^{-\frac{1}{2}}\exp(-\gamma^2/2u).
\end{equation}
An analogous calculation for $P(T_\gb>t+h\vert\mce_t)$ can be used to show that 
\begin{eqnarray}\label{e18}
P(T_\gb>t+h\vert\mce_t)&=&P(1-F_{T_{\gb-\ga}}(t+h-T_\ga)\vert\mce_t)\notag\\
&=&\{1-F_{T_{\gb-\ga}}(t+h-T_\ga)\}1_{\{T_\ga\leq t\}}\notag\\
&&+E(1-F_{T_{\gb-\ga}}(t+h-T_\ga)\vert\mce_t)1_{\{T_\ga>t\}}
\end{eqnarray}
Using~\eqref{e17} and~\eqref{e18} we obtain
\begin{eqnarray*}\label{e19}
E(N_{t+h}-N_t\vert\mcf_t)&=&\frac{F_{T_{\gb-\ga}}(t+h-T_\ga)-F_{T_{\gb-\ga}}(t-T_\ga)}{1-F_{T_{\gb-\ga}}(t+h-T_\ga)}1_{[T_\ga,T_\gb)}(t)\\
&&+E(F_{T_{\gb-\ga}}(t+h-T_\ga)\vert\mce_t)1_{[0,T_ga)}(t)
\end{eqnarray*}
The last term above is at most $F_{T_{\gb-\ga}}(h)1_{[0,T_\ga)}(t)$, and hence if we divide by $h$ and let $h$ tend to 0 it converges to its density at 0 which is 0 by~\eqref{e17bis}. Therefore we have that
\begin{equation}\label{e20}
\frac{1}{h}P(t+h\geq T_\gb>t\vert\mcf_t)\rightarrow\frac{f_{\gb-\ga}(t-T_\ga)}{1-F_{T_{\gb-\ga}}(t+h-T_\ga)}1_{[T_\ga,T_\gb)}(t)\equiv\gl_t
\end{equation}
as $h\rightarrow 0$. It is clear that the process $\gl$ is adapted to $\mbf$. Given this convergence, we have our candidate $\gl$ for the random intensity. To show that it actually is such, we need only to verify (for example) the hypotheses of a result of T. Aven~\cite{Aven}, or other results such as the theorem of Ethier-Kurtz~\cite{EK}, or alternatively see~\cite{JMP}. These are straightforward calculations and we omit them. 
\end{proof}

The following lemma is doubtless well known, but we do not know where there is a proof in the literature, so we include it here. Let $R$ and $T$ be two stopping times such that $P(R=T)=0$. Suppose they both have absolutely continuous compensators, and let
\begin{equation}\label{e20bis}
M_t=1_{\{t\geq T\}}-\int_0^{t\wedge T}\gl_sds\text{ and }N_t=1_{\{t\geq R\}}-\int_0^{t\wedge R}\mu_sds
\end{equation}
be the two martingales; by an abuse of language we say that the compensator intensity of $T$ is the process $\gl=(\gl_t)_{0\leq t\leq t\wedge T}$.

\begin{lemma}\label{l1}
Let $T,R$ be two stopping times with $P(T=R)=0$ with compensator intensities $\gl$ and $\mu$. Then the compensator intensity of the stopping time $S=T\wedge R$ is $\gl+\mu$. That is,
\bee\label{e20ter}
1_{\{t\geq T\wedge R\}}-\int_0^{t\wedge T\wedge R}(\gl_s+\mu_s)ds\quad \text {is a martingale}.
\eee
\end{lemma}

\begin{proof}
Let $M$ and $N$ be as given in~\eqref{e20bis}. Then 
$$
(M+N)_t=1_{\{t\geq T\}}+1_{\{t\geq R\}}-\left(\int_0^{t\wedge T}\gl_sds+\int_0^{t\wedge R}\mu_sds\right)
$$
Stopping $M+N$ at the time $T\wedge R$ gives us
$$
(M+N)^{T\wedge R}_t=1_{\{t\geq T\wedge R\}}-\left(\int_0^{t\wedge T\wedge R}(\gl_s+\mu_s)ds\right)
$$ 
if $P(T=R)=0$. Otherwise there is a jump at time $T\wedge R$ and it equals $1+1_{\{T=R<\infty\}}$. 
\end{proof}

We want to use the results of Theorem~\ref{t6} to give relatively explicit examples of strict local martingales with jumps, obtained via this projection method. We return to the family of continuous strict local martingales of the form
\bee\label{21}
dX_t=\gs(X_t)dB_t,\qquad X_0=1,
\eee
where $B$ is standard Brownian motion and we assume  the conditions~\eqref{e6} and~\eqref{e7}. That is, we assume $\int_0^\gep\frac{x}{\gs(x)^2}ds=\infty$ and $\int_\gep^\infty\frac{x}{\gs(x)^2}ds<\infty$ for some $\gep>0$. We will prove the following result. 

\begin{theorem}\label{t7}
Assume given a Brownian motion $B$ on the space $(\Omega,\mcg,P,\mbg)$ and a set $\gL\subset\mbr_+$, such that there exists at least one sequence of left isolated points increasing to $\infty$ and at least one other sequence of right isolated points tending to $-\infty$. We define $\mbf$ as we did before as in~\eqref{e2}. Let $X$ be the solution of~\eqref{21} with  conditions~\eqref{e6} and~\eqref{e7} holding. Moreover assume $\gs>0$, $\gs\in\mcc^2$, and both of $\gs$ and $\gs\gs^\prime$ are Lipschitz continuous. Let $M=^oX$ be its projection onto $\mbf$. Then $M$ is a strict local martingale, and its jumps have absolutely continuous compensators. 
\end{theorem}

\begin{proof}
We already have seen that $M$ is a strict local martingale, provided we can show it is a local martingale. And once this is shown, it follows from Theorem~\ref{t6} that the compensators of both the jumps and also of the jump times are absolutely continuous. To show $M$ is a local martingale, we need only to prove there exists a reducing sequence of stopping times. Let $\ga_n$ be a sequence of left isolated points in $\gL$ tending to $\infty$. Then $B_{t\wedge T_{\ga_n}}$ is bounded above by $n$, but it is not necessarily from below. We define $W=-B$ and we let $R_{\ga_n}$ be the corresponding first passage time for $W$. Then again by Theorem~\ref{t6} we have that the compensator of $R_{\ga_n}$ is absolutely continuous. Let us suppress the $n$. We know the existence of $\mbf$ adapted processes $\gl$ and $\mu$ such that $1_{\{t\geq T_\ga\}}-\int_0^t\gl_sds$ and $1_{\{t\geq R_\ga\}}-\int_0^t\mu_sds$ are martingales. But then by Lemma~\ref{l1} we have that the compensator of $T\wedge R$ is also absolutely continuous; indeed, by Lemma~\ref{l1} we even know that its form is $\int_0^{t\wedge T\wedge R}(\gl_s+\mu_s)ds$. 

From this we know that $S_{\ga_n}\equiv T_{\ga_n}\wedge R_{\ga_n}$ and that 
$$
S_{\ga_n}=\inf_{\{t>0\}}\{\vert B_t\vert\geq\ga_n\},
$$
and also that the times $(S_{\ga_n})_{n\geq 1}$ form an increasing sequence of  $\mbf$ totally inaccessible stopping times. The result now follows by Theorem~\ref{t2bis}.
\end{proof}

\subsection{A Family of Examples}\label{family}

We have worked out in some detail one example where we are able to analyze several features of the example, such as the structure of the compensators of the jumps of the local martingale. A more general family of examples can be constructed, following the same ideas, but with the cost of a slightly less explicit analysis of its features. 

We let $(\Omega,\mcg,P,\mbg)$ be a filtered complete probability space that supports a standard Brownian motion $B$. We let $\tau_1,\tau_2,\tau_3,\dots$ denote a sequence of stopping of stopping times with $\tau_i<\tau_{i+1}$ for each $i\geq 1$, and such that $\lim_{i\rightarrow\infty}\tau_i=\infty$. Moreover we assume that the times $\tau_i$ all have continuous distributions, and that there is a subsequence $\tau_{i_n}$ such that $\vert B_{t\wedge\tau_{i_n}}\vert\leq n$. Finally, let $H$ be an arbitrary predictable process such that the It\^o integral $\int_0^tH_sdB_s$ exists for each $t>0$. To eliminate trivialities, assume $H\not\equiv 0$.

\begin{theorem}\label{t8}
Let  $(\tau_i)_{i\geq 1}$ be a  sequence of stopping times, strictly increasing with continuous distributions, as defined above. Next let 
$$
J_t=\sum_{i=1}^\infty 1_{(\tau_{2i-1},\tau_{2i}]}(t)
$$
and define a subfiltration $\mbf$ of $\mbg$ by $\mcf_t=\gs(\int_0^sH_rJ_rdB_r; s\leq t)$. Then the filtration $\mbf$ ``jumps'' in the sense of Definition~\ref{d1} from $\mcf_{\tau_{2i}}$ to $\mcf_{\tau_{2i+1}}$, and the times $\tau_{2i-1}$ are totally inaccessible, for $i\geq 1$. Let $X$ be the solution of the stochastic differential equation given in~\eqref{e5},\eqref{e6},\eqref{e7}, and let $M=^oX$, the optional projection of $X$ onto $\mbf$. Then $M$ is a strict local martingale with totally inaccessible jumps at the times $(\tau_{2i-1})_{i\geq 1}$.
\end{theorem}

\begin{proof}
We know that $X$ is a strict local martingale, and that $M$ will be one, too, as soon as there exists a reducing sequence of stopping times in $\mbf$. However such a sequence exists as a consequence of Theorem~\ref{t2bis}, by the same argument using the theorem of Doss that was used in the proof of Theorem~\ref{t2bis}.

We know that the stochastic integral $HJ\cdot B$ has the same intervals of constancy as its quadratic variation process (see for example~\cite{PP}). But 
$$
[HJ\cdot B,HJ\cdot B]_t=\int_0^t(H_sJ_s)^2ds
$$
and this is constant on the stochastic intervals $(\tau_{2i},\tau_{2i+1}]$ for $i\geq 1$, and from this we can conclude that the filtration $\mbf$ jumps from $\mcf_{\tau_{2i}}$ to $\mcf_{\tau_{2i+1}}$, and from that it follows that each $\tau_{2i-1}$ is a totally inaccessible stopping time in $\mbf$. (Note that all the $\tau_i$ are predictable stopping times in $\mbg$.) The jump at $\tau_{2i+1}$ equals $\gD M_{\tau_{2i-1}}=M_{\tau_{2i+1}}-M_{\tau_{2i}}$ since the filtration itself jumps between those two stopping times. 
\end{proof}

\section{Connections to Mathematical Finance}

One thinks of a filtration as the collection of observable events evolving with time. In finance, it is natural to model different players as being more informed or less informed, and this can be modeled with filtration shrinkage and enlargement. In Fontana, Jeanblanc, and Song~\cite{FJS} which is concerned with filtration enlargement and insider trading, strict local martingales arise naturally as local martingale deflators. The idea of projecting a diffusion onto a subfiltration was used in Jarrow, Protter, and Sezer~\cite{JPSezer} in relation to reduced form credit risk models, and related in an abstract way to structural versus reduced form models in credit risk in Jarrow and Protter~\cite{BobPhilip}. Both enlargement and shrinkage are related to the preservation or loss of the absence of arbitrage opportunities of the models: see for example F\"ollmer and Protter~\cite{FP}, and Larsson~\cite{Martin}. Related to this is the concept of ``illusory arbitrage,'' also involving strict local martingales: see Jarrow and Protter~\cite{JarrowP}. One of the most commonly used applications of local martingales is their relation to mathematical models of financial bubbles, see for example~\cite{CH},\cite{JPShimbo},\cite{KKN},\cite{PhilipToBe}.

Until recently, models of mathematical bubbles (eg,~\cite{CH},\cite{JPShimbo}) were restricted to processes with continuous sample paths. However Kardaras, Kreher, and Nikeghbali~\cite{KKN} explicitly treat the more general case which includes c\`adl\`ag strict local martingales. They further use a Bessel (3) process and its reciprocal, the inverse Bessel process, to construct a strict local martingale with jumps. They do this by discretizing time in a non-random manner for the third component of the Bessel (3) process, and take an optional projection. Thus their approach is essentially the same as the one used here. Also, in the paper of Biagini, F\"ollmer, and Nedelcu~\cite{BFN}, which considers the birth of bubbles via a flow of changes of risk neutral measures, the general case with c\`adl\`ag strict local martingales is also treated. The interesting paper of H. Hulley~\cite{HH} deals with continuous path processes, but much of it makes sense for c\`adl\`ag path processes, and represents potential applications of strict local martingales with jumps in the theory of Mathematical Finance.

One can imagine the family of examples considered here arising when a trader or the market does not have all the observations $\mbg$ needed for the true underlying model (such as~\eqref{e5},\eqref{e6},\eqref{e7}), and instead has a smaller filtration $\mbf$. So what the trader sees is a projection of the underlying model onto the subfiltration. It is still a bubble in $\mbf$, and in this case the trader sees jumps, even if the underlying model is a continuous diffusion.

Another use of strict local martingales in Mathematical Finance is that of deflators that arises in stochastic portfolio theory, as can be seen in the works of Fernholz and Karatzas~\cite{FK}, Ruf~\cite{Ruf}, and the benchmark approach of Platen~\cite{Platen} and Platen and Heath~\cite{PH}. Indeed, for a probability measure $P$, let $M$ be a nonnegative uniformly integrable martingale without jumps to zero but that nevertheless reaches 0 with positive probability. Suppose $Q$ is given by $dQ=M_\infty dP$, so that $Q\ll P$ but $Q$ is not equivalent to $P$. Then under $Q$ the process $1/M$ is a strict local martingale. While $M$ can take on the value 0 and we are dividing by it, under $Q$ this happens only with $Q$-probability zero.  A nice analysis of this situation and how it can arise is presented in the recent work of Kreher and Nikeghbali~\cite{KN}.

\subsection{The Issues of the Pricing Grid and Transaction Times}

For stocks, prices in the United States markets are quoted in pennies. This means
that even if a price process is modeled as a continuous process (for example a
diffusion), it can be observed only at a grid of prices where grid points are separated
by pennies (ie, at most 1\textcent). This naturally creates a situation of filtration shrinkage, where
one observes the process only at the times it crosses the grid of prices separated
by penny units. This is, again, similar in spirit to the approach pioneered by A. Deniz
Sezer in~\cite{JPSezer} and~\cite{Deniz}. 

The other issue is that of transaction times. A common interpretation of models in
Mathematical Finance is that a price process evolves continuously, for example following
a diffusion. But one can observe the price process only at the random times when a transaction
takes place. Thus one observes the process at a well ordered sequence of stopping
times, the times when trades occur. It is typically assumed that nevertheless one
``knows" the price process at all times, especially so if the transaction times occur with
high frequency, a more common event in the modern era with the presence of high
frequency trading and ultra high frequency trading. However this is a small leap,
and it is more precise to model the information one has by the filtration obtained
by seeing the process only at the transaction times, a framework amenable to the
``family of examples" given in Paragraph~\ref{family}. Therefore these models of filtration
shrinkage enable one to model the connection between continuous time models and
``discrete time" models, via a filtration shrinkage corresponding to the transaction
times and the pricing grid crossing times. This helps to give a more precise meaning
to the ideas expressed in~\cite{BobPhilip}.

\end{document}